\documentclass{amsproc}
\usepackage{amsmath}
\usepackage{amsfonts}
\usepackage{amssymb}
\usepackage{hyperref}
\hypersetup{colorlinks,
citecolor=black,
filecolor=black,
linkcolor=black,
urlcolor=black}
\usepackage{graphics,epstopdf}
\usepackage[pdftex]{graphicx}
\usepackage{newlfont}\newlength{\defbaselineskip}
\usepackage[left=1in,right=1in,top=1in,bottom=0.8in,footskip=0.25in]{geometry}
\usepackage{setspace}
\allowdisplaybreaks
\newtheorem{theorem}{Theorem}[section]
\newtheorem{example}{Example}[section]
\newtheorem{lemma}{Lemma}[section]

\newtheorem{remark}{Remark}[section]

\numberwithin{equation}{section}

\usepackage{cite}

\usepackage{fancyhdr}
\usepackage{graphics}
\usepackage{color,colortbl}
\usepackage{rotating}
\usepackage{fancybox}
\usepackage[table]{xcolor}

\begin{document}
\title{A study on summation-integral type operators 
}
\maketitle
\begin{center}
{\bf Rishikesh Yadav$^{1,\dag}$,  Ramakanta Meher$^{1,\star}$,  Vishnu Narayan Mishra$^{2,\circledast}$}\\
$^{1}$Applied Mathematics and Humanities Department,
Sardar Vallabhbhai National Institute of Technology Surat, Surat-395 007 (Gujarat), India.\\
$^{2}$Department of Mathematics, Indira Gandhi National Tribal University, Lalpur, Amarkantak-484 887, Anuppur, Madhya Pradesh, India\\
\end{center}
\begin{center}
$^\dag$rishikesh2506@gmail.com,  $^\star$meher\_ramakanta@yahoo.com,
 $^\circledast$vishnunarayanmishra@gmail.com
\end{center}

\vskip0.5in

\begin{abstract}
In this paper,we investigate the approximation properties of the summation-integral type operators as defined by Mishra et al. (Boll. Unione Mat. Ital. (2016) 8:297-305) and determine the local results as well as prove the convergence theorem of the defined operators. Further check the asymptotic behaviour of the operators and obtain the asymptotic formula for the operators, moreover, the quantitative means of  Voronovskaja type theorem is also discussed for an upper bound of the point wise convergent, as well as obtain the Gr$\ddot{\text{u}}$ss Voronovskaya-type theorem. Finally, the graphical representation is given to support the approximation results of the operators. 
\end{abstract}
\subjclass \textbf{MSC 2010}: {41A25, 41A35, 41A36}.

\section{Introduction}
In 2016, Mishra et al. \cite{MGN} carried out their works on approximation properties for the operators defined by:
\begin{eqnarray}\label{rb}
S_n^*(g;x)=u_n\sum\limits_{j=0}^\infty s_{u_n,j}(x)\int\limits_0^{\infty} s_{u_n,j}(t) g(t) dt,
\end{eqnarray}
where $s_{u_n,j}(x)=e^{-u_nx}\frac{(u_nx)^{j}}{j!}$ and they consider $u_n\to\infty$ as $n\to\infty$ with condition $u_1=1$, here fucntions considered to be Lebesgue integrable. By simple calculation, if we take $u_n=n$ then the above operators (\ref{rb}), reduced into  Sz$\acute{\text{a}}$sz-Mirakjan Durremeyr operators defined by Mazhar and Totik \cite{MS}. The important properties of the defined operators are studied by Mishra et al. which can be applied to the operators defined Mazhar and Totik. Regarding approximation of the function by Durrmeyer type operators, many works have been done in this direction \cite{VG1,VG2, VG3}.  \\

Also the discussion regarding Durrmeyer-type modification of Sz$\acute{\text{a}}$sz-Mirakjan operators is seen in \cite{MS} where the authors gave an important result for the Durrmeyer type operators, which are defined on $[0,\infty)$ as: 
\begin{equation}\label{mt}
A_n(f;x)=f(0)s_{n,0}(x)+n\sum\limits_{j=1}^\infty s_{n,j}(x)\int\limits_0^{\infty} s_{n,j-1}(t) g(t) dt.
\end{equation}
All the above operators (\ref{rb}, \ref{mt}) are generalized form of  the Sz\'asz-Mirakjan operators \cite{MG,SO} define by:
\begin{eqnarray}\label{mt1}
{SM}_n(f;x)=\sum\limits_{j=0}^\infty s_{n,j}(x)g\left(\frac{j}{n}\right),
\end{eqnarray}
where $s_{n,j}=e^{-nx}\frac{(nx)^{j}}{j!}$ is the Sz\'asz-Mirakjan basis function. And the above operators $\{{SM}_n\}$ were studied by  Sz\'asz \cite{SO}. Also a natural generalization of the Sz\'asz-Mirakjan operators can be seen (presented in \cite{GDM}) in the form of strictly increasing sequence as a simple replacement of $n$ by $u_n$ such that $u_1=1$ and $u_n\to\infty$ as $n\to\infty$ in the above operators (\ref{mt1}) and hence the modification is seen in the form of:
\begin{eqnarray}\label{mt2}
{S}_n^*(f;x)=\sum\limits_{j=0}^\infty s_{u_n,j}(x)g\left(\frac{j}{u_n}\right).
\end{eqnarray}
Thus, the works done by Mishra et al. \cite{MGN} on the natural modification of the operators defined in \cite{MS} called as Sz$\acute{\text{a}}$sz-Mirakjan Durremeyr operators that put a crucial impact. For further proceed, we need some basic lemma.
\begin{lemma}\label{lem1}
Consider the function  $g$ is integrable, continuous and bounded on the given interval $[0,\infty)$, then the central moments can be obtained as:
\begin{eqnarray}
\Theta_{n,m}=u_n\sum\limits_{j=0}^\infty s_{u_n,j}(x)\int\limits_0^{\infty} s_{u_n,j}(t) (t-x)^m dt,
\end{eqnarray}
where $m=0,1,2,\ldots$. So for $m=0,1$, we can get the the central moments as follows:
\begin{eqnarray}
\Theta_{n,0}=1, \Theta_{n,1}=\frac{1}{u_n},
\end{eqnarray}
in general, we have
\begin{eqnarray}
u_n\Theta_{n,m+1}=x\left(\Theta_{n,m}'+2m\Theta_{n,m-1}+(1+m)\Theta_{n,m} \right),
\end{eqnarray}
this lead to 
\begin{eqnarray}
\Theta_{n,m}=O\left(u_n^{-\left[\frac{m+1}{2}\right]} \right).
\end{eqnarray}
\end{lemma}
\begin{remark}\label{rem1}
The limit of central moment can be written as:
\begin{eqnarray}
\underset{n\to\infty}\lim u_n\Theta_{n,2}=\frac{2 (1+u_n x)}{u_n}=2x.
\end{eqnarray}
\end{remark}

\begin{remark}
Second order central moment can be written as:
\begin{eqnarray}
\Theta_{n,2}=\frac{2 (1+u_n x)}{u_n^2}=\frac{2}{u_n}\left(x+\frac{1}{u_n} \right)=\frac{2}{u_n}\zeta_n(x)~(\text{say}).
\end{eqnarray}
\end{remark}
\section{Local Approximation properties}
Next, we estimate the approximation of the defined operators (\ref{rb}), by a new type of Lipschit maximal function with order $r\in(0,1]$,  defined by Lenze \cite{LB} as
\begin{eqnarray}\label{eq8}
\kappa_r(f,x)=\underset{x,s\geq 0}\sup \frac{|f(u)-f(v)|}{|u-v|^r},~~u\neq v. 
\end{eqnarray}
Using the definition Lipschit maximal function, we have the following theorem.
\begin{theorem}
For any $g\in C_B[0,\infty)$ with $r\in(0,1]$ then one can obtain
\begin{eqnarray*}
\left|S_n^*(g;x)(g;x)-g(x)\right| &\leq  \kappa_r(g,x)\left(\Theta_{n,2}\right)^{\frac{r}{2}}.
\end{eqnarray*}
\end{theorem}
\begin{proof}
By equation (\ref{eq8}), we can write
\begin{eqnarray*}
\left|S_n^*(g;x)-g(x)\right| &\leq \kappa_r(g,x)S_n^*(|u-v|^r;x).
\end{eqnarray*}
Using, H$\ddot{\text{o}}$lder's inequality with $j=\frac{2}{r}$, $l=\frac{2}{2-r}$, one can get
\begin{eqnarray*}
\left|S_n^*(g;x)(g;x)-g(x)\right| &\leq & \kappa_r(g,x)\left(S_n^*(g;x)((u-v)^2;x)\right)^{\frac{r}{2}}=\kappa_r(f,x)\left(\Theta_{n,2}\right)^{\frac{r}{2}}.
\end{eqnarray*}
\end{proof}

Next theorem is based on modified Lipschitz type spaces \cite{OMA} and this spaces is defined by 
\begin{eqnarray*}
Lip_M^{a_1,a_2}(s)=\Bigg\{ g\in C_B[0,\infty):|g(u)-g(v)|\leq M\frac{|u-v|^s}{\left(u+v^2a_1+v a_2\right)^{\frac{s}{2}}},~~\text{where}~u,v\geq0 ~\text{are~variables},~s\in(0,1] \Bigg\},
\end{eqnarray*}
here, $a_1, a_2$ are the fixed numbers and $M>0$ is constant.
\begin{theorem}
For $g\in Lip_M^{a_1,a_2}(s)$ and $0<s\leq 1$, an inequality holds:
\begin{eqnarray*}
\left|S_n^*(g;x)-g(x)\right| &\leq & M\left(\frac{\Theta_{n,2}}{x(xa_1+a_2)}\right)^{\frac{s}{2}},~~M>0.
\end{eqnarray*}
\end{theorem}

\begin{proof}
Since $s\in(0,1]$. To prove the above result, one can discuss by the cases on $s$.\\
\textbf{Case 1.} For $s=1$, it can be observed that  $\frac{1}{t+x^2a_1+xa_2)}\leq \frac{1}{x(xa_1+a_2)}$ and it implies 
\begin{eqnarray*}
\left|S_n^*(g;x)-g(x)\right| &\leq & S_n^*(|g(t)-g(x)|;x)\\
&\leq & M S_n^*\left(\frac{|t-x|}{\left(t+x^2a_1+xa_2\right)^{\frac{1}{2}}};x\right)\\
&\leq & \frac{M}{\left(x(xa_1+a_2)\right)^{\frac{1}{2}}}S_n^*(|t-x|;x) \\
&\leq & \frac{M}{\left(x(xa_1+a_2)\right)^{\frac{1}{2}}}\left(\Theta_{n,2}\right)^{\frac{1}{2}}\\
&\leq & M\left(\frac{\Theta_{n,2}}{x(xa_1+a_2)}\right)^{\frac{1}{2}}.
\end{eqnarray*}
\textbf{Case 2.} for $s\in (0,1)$ and using  H$\ddot{\text{o}}$lder inequality with $l=\frac{2}{s}, m=\frac{2}{2-s}$, we get
\begin{eqnarray*}
\left|S_n^*(g;x)-g(x)\right| &\leq & \left(S_n^*(|g(t)-g(x)|^{\frac{2}{s}};x)\right)^{\frac{s}{2}}\leq MS_n^*\left(\frac{|t-x|^{2}}{\left(t+x^2a_1+xa_2\right)};x\right)^{\frac{s}{2}}\\
&\leq & MS_n^*\left(\frac{|t-x|^{2}}{\left(x(xa_1+a_2)\right)};x\right)^{\frac{s}{2}}\\
&\leq & M\left(\frac{\Theta_{n,2}}{x(xa_1+a_2)}\right)^{\frac{s}{2}}.
\end{eqnarray*}
Thus, the proof is completed. 
\end{proof}

\begin{theorem}\label{th1}
For the continuous and bounded function $g$ defined on $[0,\infty)$, the convergence of the operators can be obtained as:
\begin{eqnarray}
\underset{n\to\infty}\lim S_n^*(g;x)=g(x), 
\end{eqnarray}
uniformly on any compact interval of $[0,\infty)$.
\end{theorem}
\begin{proof}
Using Bohman-Korovkin theorem, we can get our required result. Since $\underset{n\to\infty} \lim S_n^*(1;x)\to1$, $\underset{n\to\infty} \lim S_n^*(t;x)\to x$, $\underset{n\to\infty} \lim S_n^*(t^2;x)\to x^2$ and hence the proposed operators $S_n^*(g;x)$  converge uniformly to the function $g(x)$ on any compact interval of $[0,\infty)$. 
\end{proof}

\section{Asymptotic behaviour of the operators}
To check the asymptotic behavior of the operators, we shall prove the Voronovskaaya type theorem. 
\begin{theorem}
Let us consider the function $g$ is integrable, continuous and bounded on $[0,\infty)$as well as the second derivative of the function is exists at a point $x\in[0,\infty)$, then the convergence of the operators can be obtained as:
\begin{eqnarray}
\underset{n\to\infty}\lim u_n\left(S_n^*(g;x)g(t)-g(x)\right)&=& g'(x)+xg''(x).
\end{eqnarray}
\end{theorem}
\begin{proof}
Using the Taylor's series expansion, one can write
\begin{eqnarray}\label{eq1}
g(t)-g(x)=(t-x)g'(x)+\frac{1}{2}(t-x)^2g''(x)+\zeta(t,x)(t-x)^2,
\end{eqnarray}
where $\zeta(t,x)$ be such that $\underset{t\to x}\lim \zeta(t,x)=0.$ On applying the proposed operators on the above equation (\ref{eq1}), we get
\begin{eqnarray}\label{eq2}
S_n^*(g;x)g(t)-g(x)&=& g'(x)S_n^*(t-x;x)+\frac{g''(x)}{2}S_n^*((t-x)^2;x)+S_n^*\left(\zeta(t,x)(t-x)^2\right)
\end{eqnarray}
Here
\begin{eqnarray}
S_n^*\left(\zeta(t,x)(t-x)^2\right)\leq \sqrt{S_n^*\left(\zeta^2(t,x)\right)S_n^*\left((t-x)^4\right)}
\end{eqnarray}
Using Theorem \ref{th1}, we get 

\begin{eqnarray}
\underset{n\to\infty}\lim S_n^*\left(\zeta^2(t,x)\right)=\zeta^2(x,x)=0.
\end{eqnarray}
And using Lemma \ref{lem1}, we can have 
\begin{eqnarray}
S_n^*\left((t-x)^4\right)=O\left(u_n^{-2} \right),
\end{eqnarray}
thus 
\begin{eqnarray}
\underset{n\to\infty}\lim S_n^*\left(\zeta(t,x)(t-x)^2\right)=0
\end{eqnarray}
Therefore, form equation (\ref{eq2}) and Lemma \ref{lem1}, one can write 
\begin{eqnarray}
\underset{n\to\infty}\lim u_n\left(S_n^*(g;x)g(t)-g(x)\right)&=& g'(x)+xg''(x).
\end{eqnarray}
Hence, the required result.
\end{proof}
\section{Quantitative approximation}

Generally, we check the pointwise convergence of the operators in the form of Voronovskaya-type theorem but in the quantitative means, we determine  an upper bound of this convergence. So, here we describe the quantitative means of Voronovskaya type theorem for the proposed operators. Before, proceeding on the main results, we need some functions classes,which are defined below:

$B_w[0,\infty)=\{g:[0,\infty)\to\mathbb{R} |~~  |g(x)|\leq Mw(x)~~\text{with~the~supremum~norm}~~ \|g\|_w=\underset{x\in [0,\infty)}\sup\frac{g(x)}{w(x)}<+\infty \}$, where $M>0$ is a constant depending on $f$ and the spaces
$$C_w[0,\infty)=\{g\in B_w[0,\infty), ~g~\text{is~contiuous} \},$$
$$C_w^k[0,\infty)=\{g\in C_w[0,\infty),\underset{x\to\infty}\lim\frac{|f(x)|}{w(x)}=k_g<+\infty\},$$      
where $w(x)=1+x^2$ is a weight function. Here, the weighted modulus of smoothness is defined in \cite{IN1} and is denoted by $\Delta(g;\xi)$,  given as:
\begin{eqnarray}\label{eq3}
\Delta(g;\xi)=\underset{0\leq h\leq\xi,~0\leq x\leq\infty}\sup \frac{|g(x+h)-g(x)|}{(1+h^2)(1+x^2)},~~~~~~~g\in C_w^k[0,\infty),~~~\xi>0. 
\end{eqnarray} 

The properties of the weighted modulus of smoothness are as: 
\begin{eqnarray}
\underset{\xi\to 0}\lim\Delta(g;\xi)=0
\end{eqnarray}
and
\begin{eqnarray}\label{eq4}
\Delta(g;\eta\xi)\leq 2(1+\eta)(1+\xi^2)\Delta(g;\xi),~~~\eta>0.
\end{eqnarray}
\begin{remark}
By the above relation \ref{eq4} and (\ref{eq3}), one can write:
\begin{eqnarray*}
|g(t)-g(x)|&\leq &(1+(t-x)^2)(1+x^2)\Delta(g;|t-x|)\\
&\leq & 2\left(1+\frac{|t-x|}{\xi} \right)(1+\xi^2)\Delta(g;\xi)(1+(t-x)^2)(1+x^2).
\end{eqnarray*} 
\end{remark}

\begin{theorem}\label{th4} 
For the function $g\in C_w^k[0,\infty)$ and assuming $g''(x)$ exists at a point $x$, the following inequality holds:
\begin{eqnarray}
u_n\left|S_n^*(g;x)-g(x)-\frac{g'(x)}{u_n}-\frac{g''(x)}{u_n}\left(x+\frac{1}{u_n} \right)\right| &=& O(1)\Delta\left(g'', \sqrt{\frac{1}{u_n}}\right). 
\end{eqnarray}
\end{theorem}
\begin{proof}
By Taylor's series expansion, one can obtain
\begin{eqnarray}
g(t)-g(x)=g'(x)(t-x)+\frac{g''(x)}{2}(t-x)^2+\zeta(t,x),
\end{eqnarray}
where $\zeta(t,x)=\frac{\mathrm{g}''(\theta)-\mathrm{g}''(x)}{2!}(\theta-x)^2$ and $\theta\in (t,x)$.
Applying operators (\ref{rb}) and multiplying by $u_n$ on both sides to above expansion, we obtain:
\begin{eqnarray*}\label{n1}
u_n\left|S_n^*(g;x)-g(x)-g'(x)S_n^*(t-x;x)-\frac{g''(x)}{2}S_n^*((t-x)^2;x)\right| &\leq & u_nS_n^*(|\zeta(t,x)|;x)\\
u_n\left|S_n^*(g;x)-g(x)-\frac{g'(x)}{u_n}-\frac{g''(x)}{u_n}\left(x+\frac{1}{u_n} \right)\right| &\leq & u_nS_n^*(|\zeta(t,x)|;x).
\end{eqnarray*}
On the other hand,

\begin{eqnarray*}
\frac{g''(\theta)-g''(x)}{2} &\leq &\frac{1}{2}(1+(\theta-x)^2)(1+x^2)\Delta(g'',|\theta-x|)\\
&\leq &\frac{1}{2}(1+(t-x)^2)(1+x^2)\Delta(g'',|t-x|)\\
&\leq & \left(1+\frac{|t-x|}{\delta} \right)(1+\delta^2)(1+(t-x)^2)(1+x^2)\Delta(g'',\delta),
\end{eqnarray*}
and it can be written as:
\begin{eqnarray}
\left|\frac{g''(\theta)-g''(x)}{2}\right| &\leq &  
\begin{cases}
    2(1+\delta^2)^2(1+x^2)\Delta(f'',\delta),& |t-x|<\delta,\\
    2(1+\delta^2)^2(1+x^2)\frac{(t-x)^4}{\delta^4}\Delta(g'',\delta),& |t-x|\geq\delta.
\end{cases} 
\end{eqnarray}
So, for $\delta\in(0,1)$, we get
\begin{eqnarray}
\left|\frac{g''(\theta)-g''(x)}{2}\right| &\leq & 8(1+x^2)\left(1+\frac{(t-x)^4}{\delta^4}\right)\Delta(g'',\delta). 
\end{eqnarray}
Hence, $$(|\zeta(t,x)|;x)\leq 8(1+x^2)\left((t-x)^2+\frac{(t-x)^6}{\delta^4}\right)\Delta(g'',\delta).$$
Thus, applying the proposed operators (\ref{rb}) to the both sides and using the Lemma \ref{lem1}, we get  

\begin{eqnarray*}
S_n^*(| \zeta(t,x)|;x)&\leq & 8(1+x^2)\Delta(g'',\delta)\left(S_n^*((t-x)^2;x)+\frac{S_n^*((t-x)^6;x)}{\delta^4}\right) \\
&\leq & 8(1+x^2)\Delta(g'',\delta) \left(O\left(\frac{1}{u_n} \right)+\frac{1}{\delta^4} O\left(\frac{1}{u_n^3} \right) \right),~~\text{as}~u_n\to\infty.
\end{eqnarray*}
Choose, $\delta=\sqrt{\frac{1}{u_n}}$, then
\begin{eqnarray}
S_n^*(| \zeta(t,x)|;x)\leq 8 O\left(\frac{1}{u_n}\right)(1+x^2) \Delta\left(g'',\sqrt{\frac{1}{u_n}}\right).
\end{eqnarray}

Thus, it yields as:
\begin{eqnarray}\label{n2}
u_nS_n^*(|\eta(t,x)|;x)=O(1)\Delta\left(g'', \sqrt{\frac{1}{u_n}}\right).
\end{eqnarray}
By (\ref{n1}) and (\ref{n2}), we obtain the required result.
\end{proof}
\subsection{Gr$\ddot{\text{u}}$ss Voronovskaya-type Theorem}
This type of theorem plays an important role in the theory of approximation. First of all, in 1935, Gr$\ddot{\text{u}}$ss \cite{GG} defined an inequality, known as  Gr$\ddot{\text{u}}$ss inequality after his name and it estimate with a relation between the integral of a product and product of integrals of the two functions. Its interest increased after its publication, now a days, the importance of this inequality is being usually seen in many research articles. Firstly, Gal and Gonska \cite{GGH} deterimed the Gr$\ddot{\text{u}}$ss Voronovskaya-type  theorem with the aid of Gr$\ddot{\text{u}}$ss inequality  for the Bernstein's polynomials  and after that many reserachers contributed thier effort in this regard and effective research is being done in this direction, this type of research put a crucial impact for the liner positive operators  in  appproximation theory. In a note \cite{GTG1}, the authors obtained a new approach with the help of the least concave majorant by using Gr$\ddot{\text{u}}$ss inequality to the operators on a compact interval. Some research regarding Gr$\ddot{\text{u}}$ss Voronovskaya can be seen in  \cite{DEM,UAT,RY1,RY2,RY3,TJN}.

\begin{theorem}
Let $f, g\in C_w^k[0,\infty)$ for which $f', f'', g', g''\in C_w^k[0,\infty)$ then for each $x\geq0$, an expression can be obtained, which is:
\begin{eqnarray}
\underset{n\to\infty}\lim u_n\left(S_n^*(fg;x)-S_n^*(f;x)S_n^*(g;x) \right)=2xf'(x)g'(x).
\end{eqnarray}
\end{theorem}
\begin{proof}
For the function $f,g\in C_w^k[0,\infty)$ with $f', f'', g', g''\in C_w^k[0,\infty)$, we can write:
\begin{eqnarray*}
 n\left(S_n^*(fg;x)-S_n^*(f;x)S_n^*(g;x) \right)&=& n\Bigg\{\Bigg(S_n^*(fg;x)-f(x)g(x)-(fg)'\Theta_{n,1}\\
 &&-\frac{(fg)''}{2!}\Theta_{n,2}\Bigg)-g(x)\Bigg(S_n^*(f;x)-f(x)\\
 &&-f'(x)\Theta_{n,1}-\frac{f''(x)}{2!}\Theta_{n,2} \Bigg)\\
 &&-S_n^*(f;x)\Bigg(S_n^*(g;x)-g(x)-g'(x)\Theta_{n,1}\\
 &&-\frac{g''(x)}{2!}\Theta_{n,2} \Bigg)+\frac{g''(x)}{2!}S_n^*((t-x)^2;x)\\
 &&\times \left(f- S_n^*(f;x)\right)+f'(x)g'(x)\Theta_{n,2}\\
 &&+ g'(x)\Theta_{n,1}\left(f- S_n^*(f;x)\right) \Bigg\}.
\end{eqnarray*}

For sufficiently large vale of $n$, applying the Theorem \ref{th1} and with the help of Theorem \ref{th4} as well as applying the Remark \ref{rem1}, we get our desired result.
\begin{eqnarray*}
 \underset{n\to\infty}\lim u_n\left(S_n^*(fg;x)-\mathcal{U}_{n}^{[\alpha]}(f;x)S_n^*(g;x) \right)&=& 2xf'(x)g'(x).
\end{eqnarray*}

\end{proof}
\begin{example}
For the approximation by the operators defined by (\ref{rb}) to the given function, here, we consider the function $g:[0,4]\to [0,\infty)$ such that $g(x)=e^x$(blue), for all $x\in[0,4].$ Taking $n=25,50,100$ and corresponding operators are as $S_{25}^*$(green), $S_{50}^*$(red), $S_{100}^*$(black). Here the convergence can be seen by observing through the given Figure \ref{F1}. As the value of $n$ is increased, the approximation is going to be good. 
\begin{figure}[h!]
    \centering 
    \includegraphics[width=.52\textwidth]{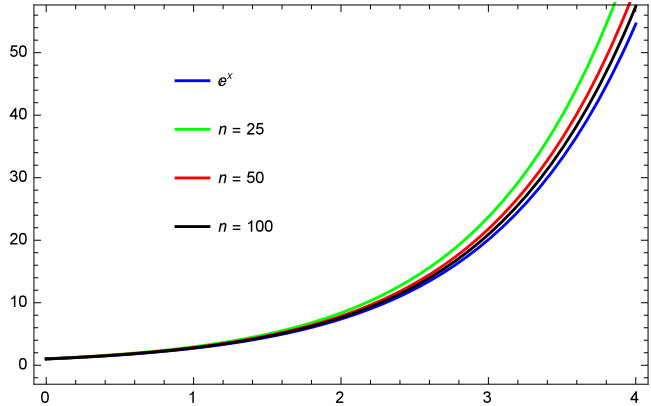}   
    \caption[Description in LOF, taken from~\cite{source}]{The convergence of the operators $S_{n}^*(g;x)$ to the function $g(x)(blue)$.}
    \label{F1}
\end{figure}
\end{example}
\begin{example}
Let the function $g=x^2\sin{2\pi x}$(blue) and for the values of $n$, equals to $50, 100, 150, 200, 300$ then the convergence of the corresponding operators $S_{50}^*, S_{100}^*, S_{150}^*,S_{200}^*, S_{300}^*$, represented pink, red, magenta, black, green colors respectively to the function is good as the value of $n$ is large as given in Figure \ref{F2}. The approximation can also be seen by observing from the given figure.

\begin{figure}[h!]
    \centering 
    \includegraphics[width=.52\textwidth]{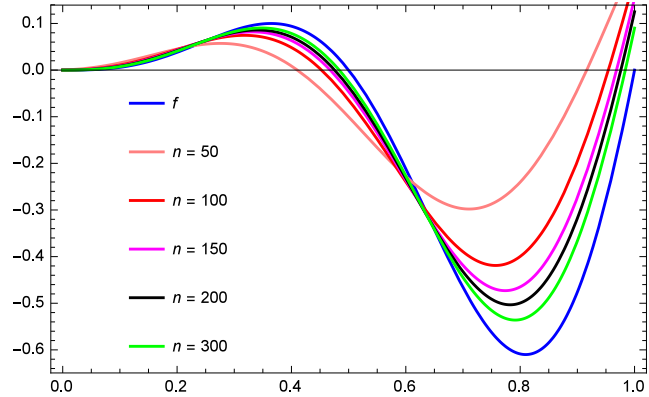}   
    \caption[Description in LOF, taken from~\cite{source}]{The convergence of the operators $S_{n}^*(g;x)$ to the function $g(x)(blue)$.}
    \label{F2}
\end{figure}
\end{example}
\textbf{Conclusion:}
Here we have determined the approximation properties for the functions belonging to different spaces and moreover, the order of approximation of the operators has been discussed. The asymptotic behaviour of the operator is discussed in the form of Voronovskaja type theorem and in this regard, we determined the upper bound of the pointwise convergent of the asymptotic formula for the operators with the prove of {Gr$\ddot{\text{u}}$ss Voronovskaya-type theorem. Finally, the approximation results of the operator are through graphically to justify the approximation properties of the operators.


\begin{thebibliography}{10}
\bibitem{DEM} Deniz EM. Quantitative estimates for Jain-Kantorovich operators. Commun. Fac. Sci. Univ. Ank. S\'er. A1 Math. Stat. 2016 Jan 1;65(2):121-32.
\bibitem{GGH} Gal S, Gonska H. Gr\" uss and Gr\" uss-Voronovskaya-type estimates for some Bernstein-type polynomials of real and complex variables. arXiv preprint arXiv:1401.6824. 2014 Jan 27.
\bibitem{GDM} Gandhi RB, Deepmala, Mishra VN. Local and global results for modified Sz\'asz–Mirakjan operators. Mathematical Methods in the Applied Sciences. 2017 May 15;40(7):2491-504.
\bibitem{GTG1} Gonska H, Tachev G. Gr$\ddot{\text{u}}$ss-type inequalities for positive linear operators with second order moduli. Matemati\c{c}ki vesnik. 2011 Jan;63(4):247-52.
\bibitem{GG} Gr$\ddot{\text{u}}$ss G. $\ddot{\text{U}}$ber das Maximum des absoluten Betrages von $\frac{1}{b-a}\int\limits_a^bf(x)g(x)dx-\frac{1}{(b-a)^2}\int\limits_a^bf(x)dx\int\limits_a^bg(x)dx$. (German) Math. Z 39(1), 215-226(1935).
\bibitem{IN1} Ispir N. Rate of convergence of generalized rational type Baskakov operators. Mathematical and computer modelling. 2007 Sep 1;46(5-6):625-31.
\bibitem{LB} Lenze B. On Lipschitz-type maximal functions and their smoothness spaces. In Indagationes Mathematicae (Proceedings) 1988 Jan 1 (Vol. 91, No. 1, pp. 53-63). North-Holland.
\bibitem{MS} Mazhar S, Totik V. Approximation by modified Sz$\acute{\text{a}}$sz-operators. Acta Scientiarum Mathematicarum. 1985 Jan 1;49(1-4):257-69.
\bibitem{MG} Mirakjan GM. Approximation of continuous functions with the aid of polynomials. In Dokl. Acad. Nauk SSSR 1941 (Vol. 31, pp. 201-205).
\bibitem{MGN} Mishra VN, Gandhi RB, Nasaireh F. Simultaneous approximation by Sz\'{a}sz–Mirakjan-Durrmeyer-type operators. Bollettino dell'Unione Matematica Italiana. 2016 Jan 1;8(4):297-305.

\bibitem{VG1} Mishra VN, Gandhi RB. Simultaneous approximation by Sz\'asz–Mirakjan–Stancu–Durrmeyer type operators. Periodica Mathematica Hungarica. 2017 Mar 1;74(1):118-27.
\bibitem{VG2} Mishra VN, Gandhi RB. A summation-integral type modification of Sz\'{a}sz-Mirakjan operators. Mathematical Methods in the Applied Sciences. 2017 Jan 15;40(1):175-82.
\bibitem{VG3} Mishra VN, Gandhi RB, Mohapatra RN. A summation-integral type modification of Sz\'asz-Mirakjan-Stancu operators. J. Numer. Anal. Approx. Theory. 2016 Sep 19;45(1):27-36.
\bibitem{OMA} $\ddot{\text{O}}$zarslan MA, Actu\v{g}lu H. Local approximation properties for certain King type operators. Filomat 27 (2013), no. 1, 173-181.
\bibitem{SO} Sz\'asz O. Generalization of S. Bernstein's polynomials to the infinite interval. J. Res. Nat. Bur. Standards. 1950 Sep;45:239-45.
\bibitem{TJN} Tariboon J, Ntouyas SK. Quantum integral inequalities on finite intervals. Journal of Inequalities and Applications. 2014 Dec 1;2014(1):121.
\bibitem{UAT} Ulusoy G, Acar T. $q$-Voronovskaya type theorems for $q$-Baskakov operators. Mathematical Methods in the Applied Sciences. 2016 Aug;39(12):3391-401.
\bibitem{RY1} Yadav R, Meher R, Mishra VN. Approximations on Stancu variant of Sz\'asz-Mirakjan-Kantorovich type operators. arXiv preprint arXiv:1911.11479. 2019 Nov 26.
\bibitem{RY2} Yadav R, Meher R, Mishra VN. Approximation on Durrmeyer modification of generalized Sz\'asz-Mirakjan operators. arXiv preprint arXiv:1911.12972. 2019 Nov 29.
\bibitem{RY3} Yadav R, Meher R, Mishra VN. Approximation properties by some modified Sz\'asz-Mirakjan-Kantorovich operators. arXiv preprint arXiv:1912.04537. 2019 Dec 10.



%
\end{thebibliography}
\end{document}